\documentclass[11pt]{article}
\usepackage{calc}
\usepackage{color}
\usepackage[english]{babel}
\usepackage{amsfonts}
\usepackage{latexsym}
\usepackage{placeins}
\usepackage{hyperref}
\let\svthefootnote\thefootnote
\newcommand\freefootnote[1]{%
	\let\thefootnote\relax%
	\footnotetext{#1}%
	\let\thefootnote\svthefootnote%
}
\usepackage{amsthm}
\usepackage{hyperref}
\usepackage{amsmath}
\usepackage{amsfonts}
\usepackage{amssymb}
\usepackage{amsthm}
\usepackage{breqn}
\usepackage{pdfpages}
\usepackage{makeidx}
\usepackage{blindtext}
\usepackage{float}
\usepackage[a4paper,top=5.5cm,bottom=4cm,left=2.5cm, right=2.5cm,foot=1cm]{geometry}
\newtheorem{teo}{Theorem}
\newtheorem{lem}[teo]{Lemma}
\newtheorem{prop}[teo]{Proposition}

\theoremstyle{remark}
\newtheorem{obss}[teo]{\bf Remark}

\usepackage{amssymb}
\usepackage{authblk}
\usepackage{amsmath}
\usepackage[cp1250]{inputenc}
\usepackage[OT4]{fontenc}

\renewcommand\Affilfont{\itshape\small}

\let\LaTeXtitle\title
\renewcommand{\title}[1]{\LaTeXtitle{\Large{\textbf{#1}}}}

\title{One-dimensional Piecewise Smooth Rational Degree Maps.}

\makeatletter
\newcommand\email[2][]%
{\newaffiltrue\let\AB@blk@and\AB@pand
	\if\relax#1\relax\def\AB@note{\AB@thenote}\else\def\AB@note{\relax}%
	\setcounter{Maxaffil}{0}\fi
	\begingroup
	\let\protect\@unexpandable@protect
	\def\thanks{\protect\thanks}\def\footnote{\protect\footnote}%
	\@temptokena=\expandafter{\AB@authors}%
	{\def\\{\protect\\\protect\Affilfont}\xdef\AB@temp{#2}}%
	\xdef\AB@authors{\the\@temptokena\AB@las\AB@au@str
		\protect\\[\affilsep]\protect\Affilfont\AB@temp}%
	\gdef\AB@las{}\gdef\AB@au@str{}%
	{\def\\{, \ignorespaces}\xdef\AB@temp{#2}}%
	\@temptokena=\expandafter{\AB@affillist}%
	\xdef\AB@affillist{\the\@temptokena \AB@affilsep
		\AB@affilnote{}\protect\Affilfont\AB@temp}%
	\endgroup
	\let\AB@affilsep\AB@affilsepx
}
\makeatother

\author[*1]{Maur\'{\i}­cio Firmino Silva Lima}
\author[2]{Tiago Rodrigo Perdig\~ao}
\affil[1]{Centro de Matemática, Computaç\~ao e Cogniç\~ao, UFABC, CEP: 09210-580, Brazil,}
\email{mauricio.lima@ufabc.edu.br}
\affil[2]{CEFET-Centro Federal tecnológico de Minas Gerais, CEP: 35.790-636, Brazil,}
\email{tiagomatt@cefetmg.br}

\begin{document}
	\pagenumbering{arabic}
\date{}		
	\maketitle

	\freefootnote{2020 Mathematics Subject Classification: 37G35, 	37Gxx, 	34C23, 34C25.}
	
	\freefootnote{Key words and phrases. Rational degree maps, piecewise-smooth maps, bifurcations, period doubling bifurcation, robust chaotic motion. }
	
\freefootnote{The author Tiago Rodrigo Perdig\~ao is a Phd student in the postgraduate program of \textit{Centro de Matemática, Computaç\~ao e Cogniç\~ao} at UFABC.} 

\freefootnote{*Corresponding author: Maur\'{\i}­cio Firmino Silva Lima.}
	\begin{section}{Abstract}
		In this paper, we consider a class of continuous maps characterized by a singularity of order $x^{q/p}$ (with $p,q \in \mathbb{N}$, $p>q$, and $(p,q)=1$) on one side of the discontinuity boundary $\Sigma$ and a linear behaviour on the other side. Such maps arise naturally in the study of grazing bifurcations of hybrid and piecewise flows. In this context the boundary collision of a fixed point of the map with $\Sigma$ then corresponds to a grazing bifurcation of the flow. We will start by studying one-dimensional maps, and the main result of this paper is a classification of all bifurcation scenarios, including: period doubling and robust chaos.
	\end{section}
	
\begin{section}{Introduction}

\hspace{0.6cm}With the evolution of research in non-linear dynamics, the theory of one-dimensional maps has played a crucial role. It was in this context that the bifurcation and chaos resulting from period doubling were initially described in \cite{reference5}. When addressing the chaotic dynamics of these maps, the analysis of linear maps in one-dimensional parts, with a single transition point, gains prominence, being commonly known as tent maps. These maps are often used as simple and explicit examples in calculations.

The scientific interest in piecewise smooth dynamics has grown significantly in recent years, especially due to its relevance in applied problems. We highlight some important works: in \cite{reference7} the authors study the movement of an oscillator imposed with a single degree of freedom, subject to an amplitude restriction. In this context, they use analytical methods in order to examine the singularities caused by the impact of a grazing orbit.

In the study of case VI presented in \cite{reference2}, such maps emerge naturally in the problem. Furthermore, these maps appear as Poincaré Maps associated with impact oscillator systems with multiple impacts (see Chapter 6 of \cite{reference2}). 

In \cite{reference8} the authors present a classification of boundary difference bifurcations in discontinuous one-dimensional maps, depending on the configurations of the piecewise linear approximation in a neighborhood of the discontinuity point. Furthermore, a specific example of a system that generates a discontinuous map is studied by considering a well-known power electronic circuit with boost converter controlled by the current mode. 

In work \cite{reference9}, such maps appear in simple neuron firing models.

In \cite{reference1}, the authors address the dynamics and bifurcations of a family of parametrized interval maps, which exhibit a single jump discontinuity. The research shows that such maps exclusively present periodic orbits with periods of $n$, $n + 1$, $2n$ and $2n + 2$, with at least one of these orbits being attractive.

The piecewise linear maps are simpler examples than the square root map studied in \cite{reference2}. Some works in this direction where many bifurcation scenarios are considered can be found in \cite{reference4}, \cite{reference6}, \cite{reference10} and \cite{reference11}.

Our motivation is based on \cite{reference2} where the authors analyse continuous piecewise smooth maps composed by two parts, one linear part and the other of the order $\mathcal{O}(x^\gamma),$ with $\gamma\neq 1$. More specifically, they consider the family

\begin{equation}\label{hh}
	h(x,\mu_1,\mu_2)=\left\{\begin{array}{lr}
		\nu x+\mu_1,& x\ge0,\\
		&\\
		\nu_2|x|^{\gamma}+\mu_2,&x<0.
	\end{array}\right.	
\end{equation}

Observe that \eqref{hh} is continuous when $\mu_1 = \mu_2 = \mu$. The authors investigate the bifurcation of limit cycles relative to a particular case when $\gamma=1/2,$ called square root maps, and that are given by
\begin{equation}
g(\overline{x},\overline{\mu})=\left\{\begin{array}{lr}
		g_1(\overline{x},\overline{\mu})=\nu \overline{x},&H(\overline{x},\overline{\mu})=\overline{x}-\overline{\mu}\ge0,
		\\
		&\\
		g_2(\overline{x},\overline{\mu})=\sqrt{\overline{\mu}-\overline{x}}+\nu\overline{\mu},& H(\overline{x},\overline{\mu})=\overline{x}-\overline{\mu}\le0,\\
	\end{array}\right.	
\end{equation}
with $0<\nu<1$.
\vspace{0.3cm}

This kind of map is of great interest, since that it naturally appears in the study of problems related to bifurcation in impacting hybrid systems where the region $H(\bar{x}, \bar{\mu}) <0$ is called the impacting region. Our main purpose in this paper is to study a more general family where, in one side of the discontinuity boundary $\Sigma,$ the map is linear and has a term of the order $\mathcal{O}(x^{q/p}),$ with $p,q \in \mathbb{N},\,\,\,p>q,\,\,\textnormal{and}\,\,(p,q)=1$ on the other side of $\Sigma.$

In this direction we considered a rational degree map, $f:D\times\mathbb{R}\longrightarrow\mathbb{R}$, with $D\subset\mathbb{R}$, described by 
	
		\begin{equation}
			f(\overline{x},\overline{\mu})=\left\{\begin{array}{lr}
				f_1(\overline{x},\overline{\mu})=\nu\overline{x}+\alpha\overline{\mu},& H(\overline{x},\overline{\mu})=\overline{x}-\overline{\mu}\ge0,\\
				&\\
				f_2(\overline{x},\overline{\mu})=\nu \overline{x}+\alpha\overline{\mu}+e\sqrt[p]{(\overline{\mu}-\overline{x})^q},&H(\overline{x},\overline{\mu})=\overline{x}-\overline{\mu}\le0,
			\end{array}\right.	
		\end{equation}
where, $e>0$, $0<\nu<1$, $\mu\sim0^+,$ $p,q \in \mathbb{N}$ with $p>q$ and $(p,q)=1$.
Performing the translation in the coordinate $x$ and the parameter $\mu$ given by $\overline{x}=x-\dfrac{\alpha\overline{\mu}}{\nu}$, and $\mu=\overline{\mu}\left(1+\dfrac{\alpha}{\nu}\right)$ we obtain
$$x-\mu=\overline{x}+\dfrac{\alpha\overline{\mu}}{\nu}-\overline{\mu}\left(1+\dfrac{\alpha}{\nu}\right)=\overline{x}-\overline{\mu}.$$
Therefore, in the new $(x,\mu)$--coordinates the map $f$ becomes
	\begin{equation}\label{aplicationf}
	f(x,\mu)=\left\{\begin{array}{lr}
		f_1(x,\mu)=\nu x,& H(x,\mu)=x-\mu\ge0,\\
		&\\
		f_2(x,\mu)=\nu x+e\sqrt[p]{(\mu-x)^q},& H(x,\mu)=x-\mu\le0,
	\end{array}\right.
\end{equation}
where, $e>0$, $0<\nu<1$, $\mu\sim0^+,$ $p,q \in \mathbb{N}$ with $p>q$ and $(p,q)=1$.

We will study the bifurcation scenario of \eqref{aplicationf} in a neighborhood of the fixed point $x^*=0$ when $\mu^*=0.$ Observe that, associated to map \eqref{aplicationf} we have $\Sigma=\{(x,\mu)\in D\times\mathbb{R};\,\,\,H(x,\mu)=0\}=\{x=\mu\}$. More specifically, we will describe all the bifurcation scenarios that may occur in a neighborhood of the fixed point $x^*=0$ and $\mu^*=0$ when the parameter $\nu\in(0,1).$ It is important to observe that map \eqref{aplicationf} generalizes the square root maps given in Chapter 4 of \cite{reference2}. The main result of this paper is the following:

\begin{teo}\label{MainResult}
	Consider the one-dimensional family of piecewise smooth rational degree map given by \eqref{aplicationf}. This family admits, for $\nu\in(0,1),$ a stable fixed point at the origin for $\mu\leq0.$ Related to the dynamics of $f$ for $\mu\sim0^{+}$ we have:
	\begin{enumerate}
		\item If $\dfrac{p}{p+q}<\nu<1$, then there exists a robust chaotic motion close to the origin for all sufficiently small $\mu\sim0^+$.
		\item If $1-\dfrac{q(p+q)^{(p-q)/q}}{p^{\frac{p}{q}}}<\nu<\dfrac{p}{p+q},$
		then for all  sufficiently small $\mu\sim0^+$, family \eqref{aplicationf} performs a period doubling bifurcation at $$\left(\overline{z},\mu_{PD}\right)=\left(\dfrac{p}{p+q},e^{\frac{p}{p-q}}(p+q)\left(\dfrac{q^q}{p^p}\right)^{\frac{1}{p-q}}\nu^{\frac{p }{p-q}(M-1)} \right),$$
		where $M$ is the smallest natural number satisfying $f^{i}(\nu\mu,\mu)>\mu,\,\,\forall\;i\in \{1,2,..,M-1\}$ and $f^M(\nu\mu,\mu)<\mu$.

Moreover, the $M$--periodic orbit exist in terms of $\mu$ (associated with $(M-1)$ iterations in the linear part of the map $f$ and and one iteration in the non-linear part of this same map), that is, in terms of the parameter $\mu$, the stable fixed point of $G^{[0]}=G$ (given by the map \ref{auxG}) exists $\mu\in I_M$ where
$$\mu\in I_M=\left(\mu_{PD},\mu_1\right]=\left(e^{\frac{p}{p-q}}(p+q)\left(\dfrac{q^q}{p^p}\right)^{\frac{1}{p-q}}\nu^{\frac{p }{p-q}(M-1)},e^{\frac{q}{p-q}}(1-\nu)^{\frac{q}{p-q}}\nu^{\frac{p}{p-q}(M-2)} \right];$$	
	
		\item If $0<\nu<1-\dfrac{q(p+q)^{(p-q)/q}}{p^{\frac{p}{q}}}$, then all the periodic orbit of \eqref{aplicationf} are stable.
	\end{enumerate}	
\end{teo}

\section{Preliminaries Results}
		
In this section we will present some preliminaries properties of family \eqref{aplicationf}. These properties will be useful for proving Theorem \ref{MainResult}.

For $\mu\sim 0^+$, consider the trapping region
\begin{equation}
W=\{x;\;\nu\mu\leq x\leq\mu\}.\label{trep}
\end{equation}
Also consider, for $\mu\sim 0^+$ the regions
\begin{equation}
\begin{array}{l}
		R_I=\{x;\,\,x\leq\mu\},\qquad \mbox{and}\qquad		
		R_{II}=\{x;\,\,x\geq\mu\}.
\end{array}\label{R12}
\end{equation}		
Observe that $W\subset R_I.$

Next proposition shows that from points of $W$ we can reach, by the action of map \eqref{aplicationf}, both regions $R_I$ and $R_{II}$.	

\begin{prop} Consider family \eqref{aplicationf} with $\mu\sim0^+$, $e>0$ and $0<\nu<1$. Then, there exist $\overline{x}_0,\,\, x_0\in W$ such that $\overline{x}_1=f(\overline{x}_0,\mu)\in R_{I}$ and $x_1=f(x_0,\mu)\in R_{II}.$ Moreover, for $x_0 \in W$ with $x_1 \in R_{II}$, there is $m(x_0,\mu)$ such that $x_m=f^m(x_0,\mu) \in W$.\label{proposicaom}
\end{prop}
\begin{proof}
Note that for $\overline{x}_0=\mu$ we have
$$f(\mu,\mu)=f_1(\mu,\mu)=f_2(\mu,\mu)=\nu\mu<\mu.$$
Then, for $\overline{x}_0<\mu$ with $\overline{x}_0$ suficiently close to $\mu$, we have
$$\overline{x}_1=f(\overline{x}_0,\mu)=\nu \overline{x}_0+e\sqrt[p]{(\mu-\overline{x}_0)^q}\approx\nu\mu<\mu.$$
Therefore, from the continuity of $f,$ we have that in a neighborhood of $\mu$ denoted by $B_{\overline{\delta}}(\mu)$ we get
$$f(\overline{x}_0,\mu)<\mu, \forall\; \overline{x}_0 \in B_{\overline{\delta}}(\mu), \textnormal{with}\,\,\bar{x}_0<\mu.$$
On the other hand
$$f(\nu\mu,\mu)=\nu^2\mu+e\sqrt[p]{\mu^q}\sqrt[p]{(1-\nu)^q}.$$
In order to $f(\nu\mu,\mu)>\mu$ we must have
\begin{equation}
	e>\dfrac{\mu-\mu\nu^2}{\mu^{q/p}\sqrt[p]{(1-\nu)^q}}=\dfrac{\mu^{\frac{p-q}{p}}(1-\nu^2)}{\sqrt[p]{(1-\nu)^q}}.
	\end{equation}
Previous inequality guarantees that for $e>0$ fixed, $p>q,\,\,p,q\in\mathbb{N}$ and $\mu\sim0^{+}$ sufficiently small we get $f(\nu\mu, \mu)>\mu.$ This means that $f(\nu\mu, \mu) \in R_{II}$. 

Again, by the continuity of $f,$ there is $\delta>0$ such that if $x_0\in B_{\delta}(\nu\mu)$ and $x_0>\nu\mu,$ we have $x_1=f(x_0,\mu)>\mu,$ that means that $x_1\in R_{II}$.

Now, in $R_{II}$, function \eqref{aplicationf} is linear. So, for $x_0\in R_{I}$ with $x_1=f(x_0,\mu) \in R_{II}$ we have:
\begin{equation}\label{diagrama}
	\begin{array}{rl}
		x_1=f(x_0,\mu)=f_2(x_0,\mu)&=\nu x_0+ e\sqrt[p]{(\mu-x_0)^q},\\
		&\\
		x_2 =f(x_1,\mu)=f_1(x_1,\mu)&=\nu^2x_0+\nu e\sqrt[p]{(\mu-x_0)^q},\\
		&\\
		x_3=f(x_2,\mu)=f_1(x_2,\mu)&=\nu^3x_0+\nu^2 e\sqrt[p]{(\mu-x_0)^q},\\
		&\vdots\\
		x_m=f(x_{m-1},\mu)=f_1(x_{m-1},\mu)&=\nu^{m}x_0+\nu^{m-1} e\sqrt[p]{(\mu-x_0)^q}.
	\end{array}
\end{equation}
As in $R_{II}$ the dynamics is contractive ($0<\nu<1$), it is possible to take $m=m(x_0,\mu)$ sufficiently large, such that $f^m(x_0,\mu)=\nu^{m}x_0+\nu^{m-1} e\sqrt[p]{(\mu-x_0)^q} \in W$.

We define $m(x_0,\mu)$ as the first positive iteration of $x_0 \in W$, such that $f(x_0,\mu) \in R_{II}$, and $x_m=f^m(x_0,\mu)\in W.$
\end{proof}
			\begin{figure}[H]
			\centering
			\includegraphics[width=0.8\textwidth]{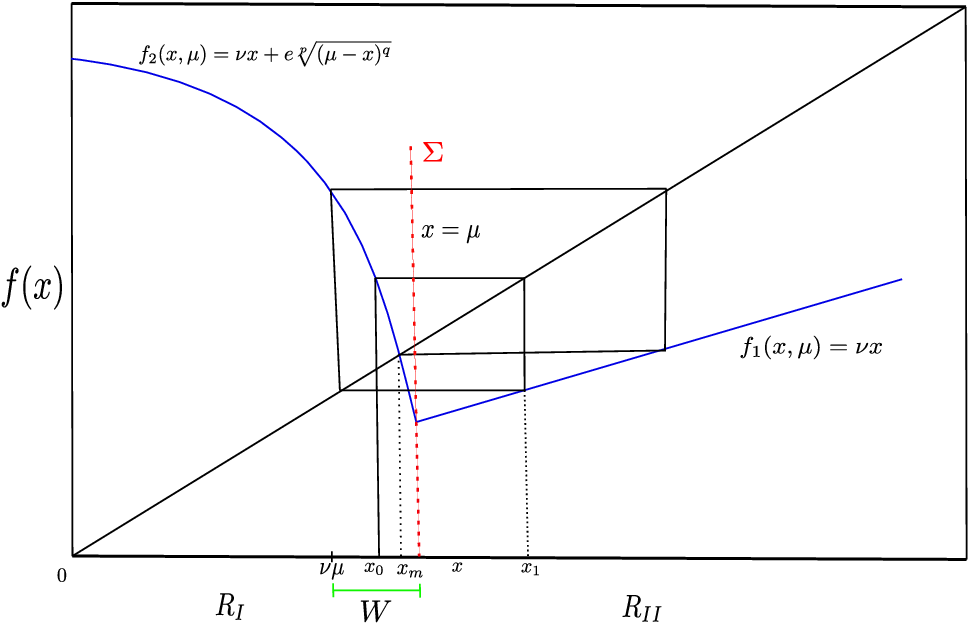}
			\caption{Dynamics of a point $x_0 \in W$ satisfying $x_1=f(x_0,\mu)\in R_{II}$ and $x_m=f^m(x_0,\mu)\in W$ where $m=m(x_0,\mu).$}
		\end{figure}

\begin{obss}\label{remF}
Observe that the trapping region $W$ given in \eqref{trep} is such that $W=[\nu\mu,\mu]=\mu[\nu,1].$ Therefore we can write a point $x_0 \in W$ into the form $x_0=\mu z$, with $z \in [\nu,1]$. Using this notation we induce from \eqref{diagrama}  the map
	\begin{equation}\label{F}
	\begin{array}{rll}
		F: [\nu,1] &\longrightarrow& W\\
		z & \mapsto &x_m=F(z)=F\left(\dfrac{x_0}{\mu}\right)=\nu^{m}\mu z+\nu^{m-1} e\sqrt[p]{\mu^q}\sqrt[p]{(1-z)^q},
	\end{array}
\end{equation}
with $z=\dfrac{x_0}{\mu},$ and $x_0\in W.$
\end{obss}

In which follows we will use the map $F$ given in \eqref{F} as an auxiliary function in order to study the bifurcation  scenarios that may occur in a neighborhood of the fixed point $x=0$ of $f(x,0)$ with $f$ given in \eqref{aplicationf}.

The following lemma guarantees that we can make the number of iterations $m(x_0,\mu)$ given in Proposition \ref{proposicaom} as large as we want.
		
		\begin{lem}With the notation of Proposition \ref{proposicaom}, for each fixed $\mu\sim 0^{+}$ function $m(x_0,\mu)$ is a decreasing function defined in $ [\nu\mu,\mu]$. Therefore its maximum $M$  occurs for $M(\mu)=m(\nu\mu,\mu)$. Moreover $M(\mu)$ can be taken large enough, if we take $\mu\sim0^+$ sufficiently small.\label{lemaM}
		\end{lem}
	
			\begin{proof}
				Observe that for fixed $\mu\sim 0^{+}$, the number of iterates $m$ is a function of $x_0$, that is
				$$
				\begin{array}{lcll}
					m:& [\nu\mu,\mu]\times\{\mu\}&\longrightarrow&\mathbb{N}\\
					&  (x_0,\mu)  &\mapsto        &m (x_0,\mu).
				\end{array}
				$$	
From Remark \ref{remF} we know that
$$F(z)\in W \iff \nu^{m}\mu z+\nu^{m-1} e\sqrt[p]{\mu^q}\sqrt[p]{(1-z)^q}\in [\nu\mu,\mu].$$
Now observe that there is a $\tilde{\mu} \in [\nu\mu,\mu]$, with $\tilde{\mu}\sim 0^+$ and
$$\nu^{m}\mu z+\nu^{m-1} e\sqrt[p]{\mu^q}\sqrt[p]{(1-z)^q}=\tilde{\mu},$$
if and only if
	\begin{equation}
					\nu^{m-1}\left(\nu\mu z+e\sqrt[p]{\mu^q}\sqrt[p]{(1-z)^q}\right)=\tilde{\mu}.\label{mu}
	\end{equation}
Applying logarithm in expression \eqref{mu}, we obtain
\begin{equation}
	m=1+\dfrac{1}{log \nu}\left[log(\tilde{\mu})-log\left(\nu\mu z+e\sqrt[p]{\mu^{q}}\sqrt[p]{(1-z)^q}\right)\right].
\end{equation}
Thus, taking $\mu\sim 0^+$ fixed, we have
$$m(z,\mu)=1+\dfrac{1}{log \nu}\left[log(\tilde{\mu})-log\left(\nu\mu z+e\sqrt[p]{\mu^{q}}\sqrt[p]{(1-z)^q}\right)\right],$$
that implies
\begin{equation}
	m_z(z,\mu)=-\dfrac{\left[\nu\mu-\dfrac{q}{p}e\sqrt[p]{\mu^{q}}(1-z)^{\frac{q-p}{p}}\right]}{\left[\nu\mu z+e\sqrt[p]{\mu^{q}}\sqrt[p]{(1-z)^q}\right]\log \nu\ln 10}\label{m}
\end{equation}

Now, since $\log \nu<0$ for all $0<\nu<1$, it follows by \eqref{m}
$$
\begin{array}{rl}
m_z(z,\mu)<0\iff&\left[\nu\mu-\dfrac{q}{p}e\sqrt[p]{\mu^{q}}(1-z)^{\frac{q-p}{p}}\right]<0\\
&\\
\iff&\dfrac{q}{p}e\sqrt[p]{\mu^{q}}(1-z)^{\frac{q-p}{p}}>\nu\mu\\
&\\
\iff&e>\dfrac{p\nu\mu^{\frac{p-q}{p}}(1-z)^{\frac{p-q}{p}}}{q}.
\end{array}
$$
So, for $p>q,\,\,p,q\in\mathbb{N}$, for $e>0$ fixed and $\mu\sim0^{+}$ sufficiently small we get $m_z(z,\lambda)<0$, for all $z\in [\nu,1]$, with $z=\dfrac{x_0}{\mu}$. Therefore, for each fixed $\mu\sim0^+,$ $m$ is a decreasing function of $x_0.$

Therefore, for each fixed $\mu\sim0^+,$ $m$ is a decreasing function of $x.$
				
As $m(.,\mu)$ is a decreasing function in $[\nu\mu,\mu]$, it takes its maximum value in $\nu\mu.$ So, $M(\mu) =m(\nu\mu,\mu)$ is the maximum number of iterations such that the orbit of $\nu\mu \in W$ returns to $R_{I}$ at the point $\mathrm{x}_M=f^M(\nu\mu,\mu) \in R_{I}$. 

Now in order to show that we can make $M=m(\nu\mu,\mu)$ sufficiently large we observe that
$$
\begin{array}{rl}
F(z)\in W \iff&\nu\mu\leq\nu^{m}\mu z+\nu^{m-1} e\sqrt[p]{\mu^q}\sqrt[p]{(1-z)^q}\leq\mu\\
&\\
\iff&\nu\mu\leq\nu^{m-1}\left(\nu\mu z+ e\sqrt[p]{\mu^q}\sqrt[p]{(1-z)^q}\right)\leq\mu\\
&\\
\iff&\dfrac{\nu\mu}{\nu\mu z+ e\sqrt[p]{\mu^q}\sqrt[p]{(1-z)^q}}\leq\nu^{m-1}\leq\dfrac{\mu}{\nu\mu z+ e\sqrt[p]{\mu^q}\sqrt[p]{(1-z)^q}},
\end{array}
$$
where $z=\dfrac{x_0}{\mu}$ and $x_0\in W.$ Applying logarithm to the previous inequality, we have
\begin{equation}
\dfrac{1}{\log{\nu}}\log\left({\dfrac{\nu\mu}{\nu\mu z+ e\sqrt[p]{\mu^q}\sqrt[p]{(1-z)^q}}}\right)\geq m-1\geq\dfrac{1}{\log{\nu}}\log\left({\dfrac{\mu}{\nu\mu z+ e\sqrt[p]{\mu^q}\sqrt[p]{(1-z)^q}}}\right),	
\end{equation}
that implies that
$$m\geq1+\dfrac{1}{\log{\nu}}\log\left({\dfrac{\mu}{\nu\mu z+ e\sqrt[p]{\mu^q}\sqrt[p]{(1-z)^q}}}\right).$$
Therefore,
$$
m\geq1+\dfrac{1}{\log{\nu}}\log\left({\dfrac{1}{\nu z+ e\mu^{\frac{q-p}{p}}\sqrt[p]{(1-z)^q}}}\right),
$$
or equivalently
$$
m\ge1-\dfrac{1}{\log{\nu}}\log\left(\nu z+ \dfrac{e}{\mu^{\frac{p-q}{p}}}\sqrt[p]{(1-z)^q}\right).
$$
As, $e>0$, $1-z>0$ for $z\in[\nu,1],$ and $p>q$ it follows that $m\longrightarrow\infty$ as $\mu\sim0^+$, and we have the result.

			\end{proof}
		
We now return to the study of the dynamics of $F$ given in Remark \ref{remF}. This map has the form
	\begin{equation}
	\begin{array}{rll}
		F: [\nu,1] &\longrightarrow& W\\
		z & \mapsto &x_m=F(z)=F\left(\dfrac{x_0}{\mu}\right)=\nu^{m}\mu z+\nu^{m-1} e\sqrt[p]{\mu^q}\sqrt[p]{(1-z)^q},
	\end{array}
\end{equation}
where $z=\dfrac{x_0}{\mu}$ and $x_0\in W.$

In order to simplify the study of the bifurcation scenarios of $F$ we will take $k=M-m,$ where $M$ is given in Lemma \ref{lemaM} and introduce the new parameter
\begin{equation}\label{lambda}
	\lambda=\left(\dfrac{\nu}{\dfrac{F(\nu)}{\mu}}\right)^{p/q}.	
\end{equation}
Next lemma provides a useful information on the new parameter $\lambda.$
\begin{lem}\label{lemal}
The parameter $\lambda$ given in \eqref{lambda}, belongs to the interval $\left[{\nu^{p/q}},1\right]$, that is,
	$$\lambda=\left(\dfrac{\nu}{\dfrac{F(\nu)}{\mu}}\right)^{p/q} \in \left[\nu^{p/q},1\right].$$
	\end{lem}
	\begin{proof}
As  $F(\nu) \in[\nu\mu,\mu]$ it follows that
$\nu\leq\dfrac{F(\nu)}{\mu}\leq1.$
Moreover, as $\nu>0$ we get
$$\nu^{p/q}\leq\left(\dfrac{\nu}{\dfrac{F(\nu)}{\mu}}\right)^{p/q}\leq1.$$
\end{proof}

Next lemma provides a simplification on the function $F$ in order to make the study of bifurcation scenarios more treatable.
\begin{lem}Let $F$ be given in \eqref{F}, $\lambda$ given by \eqref{lambda}, and $k=M-m$, where $M$ is given in Lemma \ref{lemaM}. Then $\dfrac{F(z)}{\mu }$ has the form
		\begin{equation}
		\dfrac{F(z)}{\mu}=\dfrac{\nu^{1-k}\sqrt[p]{(1-z)^q}}{\sqrt[p]{(1-\nu)^q}\sqrt[p]{\lambda^q}}+\nu^{M-k}\bigg(z-\nu\dfrac{\sqrt[p]{(1-z)^q}}{\sqrt[p]{(1-\nu)^q}}\bigg)=G^{[k]}(z,\lambda)+\nu^{M-k}\bigg(z-\nu\dfrac{\sqrt[p]{(1-z)^q}}{\sqrt[p]{(1-\nu)^q}}\bigg),\label{auxiliar}
	\end{equation}
for $k=0,1,2...$ and $z=\dfrac{x_0}{\mu}\in [\nu,1]$.\label{lemaG}
\end{lem}
\begin{proof}
From Lemma \ref{lemaM}, the maximum number of iterations $M(\mu)$ occurs at the point $x_0=\nu\mu.$ For this point we have $z=\dfrac{\nu\mu}{\mu}=\nu.$ Now from the first term on the right side of \eqref{auxiliar} we have
\begin{equation}
	\begin{array}{rl}
		\dfrac{\nu^{1-k}\sqrt[q]{(1-z)^p}}{\sqrt[p]{(1-\nu)^q}\sqrt[p]{\lambda^q}}=&\dfrac{1}{\mu}\dfrac{\nu^{1-k}\sqrt[p]{(1-z)^q}}{\sqrt[p]{(1-\nu)^q}}\dfrac{F(\nu)}{\nu},\\
		&\\
		=&\dfrac{1}{\mu}\dfrac{\nu^{-M+m}\sqrt[p]{(1-z)^q}}{\sqrt[p]{(1-\nu)^q}}\Bigg(\nu^M\nu\mu +\nu^{M-1}e\sqrt[p]{\mu^q}\sqrt[p]{(1-\nu)^q}\Bigg),\\
		&\\
		=&\dfrac{1}{\mu}\left(\nu^{m-1}e\sqrt[p]{\mu^q}\sqrt[p]{(1-z)^q}+\nu^{m+1}\mu\dfrac{\sqrt[p]{(1-z)^q}}{\sqrt[p]{(1-\nu)^q}}\right).\label{aux1}
	
	\end{array}
\end{equation}
On the other hand, from the second term on the right side of \eqref{auxiliar} we get
\begin{equation}
	\nu^{M-k}\bigg(z-\nu\dfrac{\sqrt[p]{(1-z)^q}}{\sqrt[p]{(1-\nu)^q}}\bigg)=\dfrac{1}{\mu}\left(\nu^{m}z\mu-\mu\nu^{m+1}\dfrac{\sqrt[p]{(1-z)^q}}{\sqrt[p]{(1-\nu)^q}}\right),\label{aux2}
\end{equation}
where $k=0,1,2....$. Now,  adding \eqref{aux1} and \eqref{aux2} we easily obtain \eqref{auxiliar}.
\end{proof}

Now consider the map
$$G:[\nu,1]\times\left[\nu^{p/q},1\right]\longrightarrow [\nu,1],$$
given by
\begin{equation}
	G(z,\lambda)=\dfrac{\nu\sqrt[p]{(1-z)^q}}{\sqrt[p]{(1-\nu)^q}\sqrt[p]{\lambda^q}}.\label{auxG}
\end{equation}
with, $z=\dfrac{x_0}{\mu}$ and $x_0\in W.$ 

\begin{obss}\label{obsG}
Note that from \eqref{auxG}, $G$ has a fixed point at the point $\overline{z},$ if and only if,
\begin{equation}
\dfrac{\overline{z}}{\sqrt[p]{(1-\overline{ z})^q}}=\dfrac{\nu}{\sqrt[p]{(1-\nu)^q}\sqrt[p]{\lambda^q}}.\label{condfixed}
\end{equation}
Moreover
\begin{equation}
	G_z(z,\lambda)=-\dfrac{q}{p}\dfrac{\nu}{\sqrt[p]{(1-\nu)^q}\sqrt[p]{(1-z)^{p-q}}\lambda^{q/p}}<0,\label{derG}
\end{equation}	
that means that $G(.,\lambda)$ is a strictly decreasing function. Also, as $G(\nu,\lambda)=\dfrac{F(\nu)}{\mu}\in[\nu,1]$ and $G(1,\lambda)=0$ it follows that $G(.,\lambda)$ has a single fixed point at $[\nu,1]$ (see Figure \ref{FiguraG}).
\end{obss}

		\begin{figure}[H]
		\centering
		\includegraphics[width=0.4\textwidth]{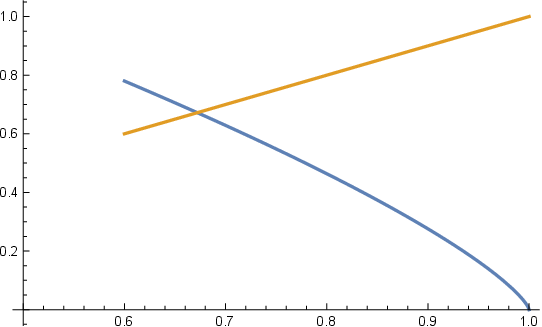}
		\caption{Graph of the identity and of function $G$ for $p=4,\;q=3,\;\nu=0.6$ and $\dfrac{1}{\lambda^{3/4}}=1.3$.}\label{FiguraG}
	\end{figure}

\begin{lem}
For all $\lambda \in\left[\nu^{p/q},1\right]$ and for all $z=\dfrac{x_0}{\mu} \in[\nu,1]$, if $\dfrac{p}{p+q}<\nu<1$, then
$$|G_z(z,\lambda)|>1.$$
Therefore, the only possible movement of the orbits of $G$ is chaotic.\label{lemachaos}
\end{lem}
\begin{proof}
Let $G$ be defined in \eqref{auxG}. Then for all $z=\dfrac{x_0}{\mu} \in[\nu_,1]$, we have
$$\nu\mu<x_0\iff \nu<\dfrac{x_0}{\mu}\iff 1-\nu<1-\dfrac{x_0}{\mu}\iff\dfrac{1}{1-\dfrac{x_0}{\mu}}>\dfrac{1}{1-\nu}.$$
Also we have
$$
\dfrac{1}{1-z}>\dfrac{1}{1-\nu}\iff\dfrac{1}{\sqrt[p]{(1-z)^{q}}}>\dfrac{1}{\sqrt[p]{(1-\nu)^{q}}}.
$$
Then, from \eqref{derG} it follows that
\begin{equation}\label{Gzv}
	\left|G_z(z,\lambda)\right|>\left|G_z(\nu,\lambda)\right|=\dfrac{q\nu}{p(1-\nu)\lambda^{q/p}}.
\end{equation}

Now for $\lambda \in[\nu^{p/q},1]$, we have $\dfrac{1}{\lambda^{q/p}}\in\left[1,\dfrac{1}{\nu}\right],$ that by \eqref{Gzv}  implies
\begin{equation}
|G_z(z,\lambda)|>\dfrac{q\nu}{p(1-\nu)\lambda^{q/p}}\in \left[\dfrac{q\nu}{p(1-\nu)},\dfrac{q}{p(1-\nu)}\right].\label{der1}
\end{equation}
Now,
\begin{equation}
\dfrac{q\nu}{p(1-\nu)}>1\iff\nu\left(p+q\right)>p\iff\nu>\dfrac{p}{p+q}.\label{der2}
\end{equation}

So, if $\dfrac{p}{p+q}<\nu<1$, it follows from \eqref{der1} and \eqref{der2} that
\begin{equation}
|G_z(z,\lambda)|>|G_z(\nu,\lambda)|=\dfrac{q\nu}{p(1-\nu)\lambda^{q/p}}>\dfrac{q\nu}{p(1-\nu)}>1,\,\,\forall \lambda\in\left[\nu^{p/q},1\right].\label{expan}
\end{equation}
This shows that the dynamics of $G(.,\lambda)$ in $[\nu,1]$ is repulsive. Therefore, $G$ has sensitive dependence on initial conditions. Moreover, we know that:

\begin{itemize}
\item[(i)] As $Im\,\,G\subset \left[\nu,1\right],$ the orbits of $G$ are bounded;

\item[(ii)] From \eqref{expan} the Lyapunov exponent of any orbit is positive and there is no asymptotically periodic orbit.

\end{itemize}

Then it follows from \textit{Definition 3.5} of \cite{reference13} that for $\dfrac{p}{p+q}<\nu<1,$ the orbit of $G$ are chaotic.
\end{proof}

Next lemma guarantees that, for $1-\dfrac{q(p+q)^{p-q/q}}{p^{\frac{p}{q}}}<\nu<1$, family \eqref{auxG} performs a period doubling bifurcation.

\begin{lem} Let $G$ be given by \eqref{auxG} with 
$1-\dfrac{q(p+q)^{p-q/q}}{p^{\frac{p}{q}}}<\nu<\dfrac{p}{p+q}.$ Then $G$ performs a period doubling bifurcation at $\left(\overline{z},\lambda_{PD}\right)=\left(\dfrac{p}{p+q},\dfrac{q(p+q)^{\frac{p-q}{q}}}{p^{\frac{p}{q}}}\dfrac{\nu^{\frac{p}{q}}}{1-\nu}\right).$\label{lemaperiod}
\end{lem}
	\begin{proof}
From \eqref{auxG} we know that
$$G(z,\lambda)=\dfrac{\nu\sqrt[p]{(1-z)^q}}{\sqrt[p]{(1-\nu)^q}\sqrt[p]{\lambda^q}}\,\,.$$
Moreover, from Remark \ref{obsG} we have that if $\overline{z}$ is a fixed point of $G$ then
\begin{equation}
\dfrac{\overline{z}}{\sqrt[p]{(1-\overline{z})^q}}=\dfrac{\nu}{\sqrt[p]{(1-\nu)^q}\sqrt[p]{\lambda^q}}.\label{fixed}
\end{equation}
Substituting \eqref{fixed} into the expression of $G_z(z,\lambda)$ given in \eqref{derG} we get
$$
G_z(\overline{z},\lambda)=-\dfrac{q\nu}{p\sqrt[p]{(1-\nu)^q}\sqrt[p]{(1-\overline{z})^{p-q}}\lambda^{q/p}}=-\dfrac{q\overline{z}}{p(1-\overline{z})^{q/p}(1-\overline{z})^{(p-q)/p}},
$$
that implies,
\begin{equation}
	G_z(\overline{z},\lambda)=-\dfrac{q\overline{z}}{p(1-\overline{z})},\label{Gx}
\end{equation}
and $G_z(\overline{z},\lambda)$ depends only on the fixed point $\overline{z}$.

Now in order to find candidates for period doubling we must study $G_z(\overline{z},\lambda)=-1$. From equation \eqref{Gx}, under this condition, we have 
\begin{equation}
\overline{z}=\dfrac{p}{p+q}.\label{zbar}
\end{equation}

Also, from \eqref{zbar} and equation \eqref{condfixed} we can find the associated bifurcation parameter $\lambda_{PD}$ by solving 
$$\dfrac{p}{p+q}=\dfrac{\nu \sqrt[p]{\Bigg(1-\frac{p}{p+q}\Bigg)^q}}{\sqrt[p]{(1-\nu)^q}\sqrt[p]{\lambda_{PD}^q}},
$$
that gives
\begin{equation}\label{L}
	\lambda_{PD}=\dfrac{q(p+q)^{\frac{p-q}{q}}\nu^{\frac{p}{q}}}{p^{\frac{p}{q}}(1-\nu)}.
\end{equation}
Moreover, as $\nu>1-\dfrac{q(p+q)^{(p-q)/q}}{p^{\frac{p}{q}}}$ a direct computation shows that $\lambda_{PD}>\nu^{p/q}.$ Therefore, from Lemma \eqref{lemal} $\lambda_{PD}$ is inside the interval of definition of the parameter $\lambda.$

Now, in order to apply Theorem \ref{FlipBifurcations} of Appendix we have to compute coefficients $K_1$ and $K_2.$

First it is not difficult to see that for all
$\nu \in \left(1-\dfrac{q(p+q)^{p-q/q}}{p^{\frac{p}{q}}}<\nu<\dfrac{p}{p+q}\right)$, we have
$$
\begin{array}{rl}
	K_1=&G_{\lambda}(\overline{z}, \lambda_{PD})G_{zz}(\overline{z}, \lambda_{PD})+2G_{\lambda z}(\overline{z}, \lambda_{PD})\\
	&\\
	=&\dfrac{1}{p^2} \left[\nu q^2(1 - \nu)^{-\frac{2q}{p}}\lambda_{PD}^{-\left(1 + 2 \frac{q}{p}\right) }\left(\dfrac{q}{p+q}\right)^{\left(-2+2\frac{q}{p}\right)}\underbrace{\left(\nu-(1-\nu)^{\frac{q}{p}}\left(\dfrac{q}{p+q}\right)^{\frac{p-q}{p}}\lambda_{PD}^{\frac{^q}{p}}\right)}_\beta\right]	\neq 0\\
\end{array}
$$

In fact, from the previous equality it follows that $K_1=0$ if and only if $\beta=0.$ However, $\beta=0$ if and only if $\lambda_{PD}=\dfrac{\nu^{\frac{p } {q }}}{(1-\nu)\left(\dfrac{q}{p+q}\right)^{\frac{p-q}{q}}}$, that does not occur due to \eqref{L}.

On the other hand, $4\nu\left(p-\dfrac{q}{2}\right)(p+q)^{\frac{1}{q}}>3\nu(p-q)(p+q)^{\frac{1}{p}},$ for all $\nu \in \left(1-\dfrac{q(p+q)^{p-q/q}}{p^{\frac{p}{q}}}, \dfrac{p}{p+q}\right),$ $p,q\in \mathbb{N}$ and $p>q$. Thus
	$$
	\begin{array}{rl}
	K_2=&\dfrac{1}{2}G_{zz}(\overline{z}, \lambda_{PD})^2+\dfrac{1}{3}G_{zzz}(\overline{z}, \lambda_{PD})^3\\
	&\\
	=&-4\nu\left(p-\dfrac{q}{2}\right)q^{\frac{q}{p}}(p+q)^{\frac{p-q}{q}}+3\nu(p-q)q^{\frac{q}{p}}(p+q)^{\frac{p-q}{p}}<0.
	\end{array}
$$
		
So it follows from Theorem \ref{FlipBifurcations} pf Appendix that there exists a smooth curve of fixed points of $G$ passing through $\left(\overline{z},\lambda_{PD}\right)=\left(\dfrac{p}{p+q},\dfrac{q(p+q)^{\frac{p-q}{q}}\nu^{\frac{p}{q}}}{p^{\frac{p}{q}}(1-\nu)}\right)$ with the fixed point changing its stability in $\left(\overline{z},\lambda_{PD}\right).$ 
Moreover, there is a smooth curve $\gamma$, passing through $\left(\overline{ z},\lambda_{PD}\right)=\left(\dfrac{p}{p+q},\dfrac{q(p+q)^{\frac{p-q}{q}}\nu^{\frac{p}{q}}}{p^{\frac{p}{q}}(1-\nu)}\right)$ such that $\gamma-\{\left(\overline{ z},\lambda_{PD}\right)\}$ is the union of stable hyperbolic 2-periodic orbits. This concludes the result.
\end{proof}
%
%
%
Next Lemma provides the value of the parameter $\mu$ associated with $\lambda_{PD}$ by equation \eqref{lambda}.

\begin{lem}\label{mu_PD}
The map $G$ given by equation \eqref{auxG} performs a period doubling bifurcation for $G\left(\overline{z},\lambda_{PD}\right)$ where $\lambda_{PD}$ and $\mu=\mu_{PD}$ related by the equation \eqref{lambda} satisfies
$$\mu=\mu_{PD}=\left(\dfrac{\nu}{\nu^{M-1}e\sqrt[p]{(1-\nu)^q}\sqrt[p]{\lambda_{PD}^q}}\right)^{\frac{p}{q-p}},$$ and $\lambda_{PD}$ is given in Lemma \eqref{lemaperiod}. Moreover, this value of parameter corresponds to
$$
\mu_{PD}=e^{\frac{p}{p-q}}(p+q)\left(\dfrac{q^q}{p^p}\right)^{\frac{1}{p-q}}\nu^{\frac{p }{p-q}(M-1)}.
$$
\begin{proof}
From \eqref{lambda} we have
$$\lambda=\left(\dfrac{\nu}{\frac{F(\nu)}{\mu}}\right)^{p/q}.$$
So the bifurcation value $\mu=\mu_{PD}$ satisfy 
$$\lambda_{PD}=\left(\dfrac{\nu}{\frac{F(\nu)}{\mu_{PD}}}\right)^{p/q},$$
and $\lambda_{PD}$ is the parameter associated with the period doubling bifurcation of map $G$, given by the Lemma \ref{lemaperiod}$\,\,$.

Substituting the expression of $F$ given in \eqref{F} and of $\lambda_{PD}$ given in \eqref{L} in the previous equation we get
$$\begin{array}{rl}
	\mu=\mu_{PD}=&\left(\dfrac{\nu}{\nu^{M-1}e\sqrt[p]{(1-\nu)^q}\sqrt[p]{\lambda_{PD}^q}}\right)^{\frac{p}{q-p}}\\
	&\\
	=&\left(\dfrac{\nu}{\nu^{M-1}e\sqrt[p]{(1-\nu)^q}\dfrac{q^{\frac{q}{p}}(p+q)^{\frac{p-q}{p}}\nu}{p\sqrt[p]{(1-\nu)^q}}}\right)^{\frac{p}{q-p}}
\end{array}.
$$
A direct computation provides
$$\mu_{PD}=e^{\frac{p}{p-q}}(p+q)\left(\dfrac{q^q}{p^p}\right)^{\frac{1}{p-q}}\nu^{\frac{p }{p-q}(M-1)}.$$
\end{proof}

\end{lem}

\begin{obss}\label{MeM-1orbit}
As we will see in the next section, a fixed point of $G$ given by
$$G=G^{[0]}=\dfrac{\nu^{1-0}\sqrt[p]{(1-z)^q}}{\sqrt[p]{(1-\nu)^q}\sqrt [p]{\lambda^q}}=\dfrac{\nu\sqrt[p]{(1-z)^q}}{\sqrt[p]{(1-\nu)^q}\sqrt [p]{\lambda^q}},$$ will be related to a $M$-periodic orbit of \eqref{aplicationf} in the sense that it is a maximal orbit of \eqref{aplicationf} having an iterate in the region where \eqref{aplicationf} is a rational degree map and $(M-1)$--iterates in the region where it is linear.

Moreover, by Lemma \ref{mu_PD}, the bifurcation parameter $\mu=\mu_{PD}$ where $G$ performs a period doubling bifurcation is given by $\lambda_{PD}$ that is associated to
$$\mu_{PD}=\mu_{PD}^0=e^{\frac{p}{p-q}}(p+q)\left(\dfrac{q^q}{p^p}\right)^{\frac{1} {p-q}}\nu^{\frac{p }{p-q}(M-1)},$$ by equation \eqref{lambda}.

In a similar way, for a $(M-k)$--periodic orbit of \eqref{aplicationf} with one iterate in the rational degree region and $M-2$ iterates in the linear region, we have $k=M-m=M-(M-k)=1$, $k=0,1,2,...$. As in the previous case, we will also see that, by equation \eqref{auxiliar} of Lemma \ref{lemaG}, this orbit will be associated with a fixed point of the map
$$G^{[k]}=\dfrac{\nu^{1-k}\sqrt[p]{(1-z)^q}}{\sqrt[p]{(1-\nu)^ q}\sqrt[p]{\lambda^q}},$$

By a direct computation, as in Lemmas \ref{lemaperiod} and \ref{mu_PD} we can see that $G^{[k]}$ performs a period doubling bifurcation for a $\lambda_{PD}^k$ that is associated to

$$\mu_{PD}^k=e^{\frac{p}{p-q}}(p+q)\left(\dfrac{q^q}{p^p }\right)^{\frac{1}{p-q}}\nu^{\frac{p }{p-q}(M-1-k)},$$  by equation \eqref{lambda}.

An important fact, that also will be verified in the proof of the main theorem in the next section, is that the qualitative behaviour of the maps
$$G^{[k]}=\dfrac{\nu^{1-k}\sqrt[p]{(1-z)^q}}{\sqrt[p]{(1-\nu)^ q}\sqrt[p]{\lambda^q}},$$
are similar for $k=0,1,2,...$.
\end{obss}

Next proposition exhibit the intervals in terms of $\mu$ on which the fixed points of $G^{[k]}$ exist. Moreover it guarantees that the intervals $I_M$ (related to the fixed point of $G^{[0]}$) and the intervals $I_{M-k}$ (related to the fixed point fo $G^{[k]},$ with $k=1,2,...$) given by Remark \ref{MeM-1orbit}, are disjoint.

\begin{prop}\label{PropMeM-1} In terms of the parameter $\mu$, the stable fixed point of $G^{[0]}=G$ exists for $\mu\in I_M$ where
	$$\mu\in I_M=\left(\mu_{PD},\mu_1\right]=\left(e^{\frac{p}{p-q}}(p+q)\left(\dfrac{q^q}{p^p}\right)^{\frac{1}{p-q}}\nu^{\frac{p }{p-q}(M-1)},e^{\frac{q}{p-q}}(1-\nu)^{\frac{q}{p-q}}\nu^{\frac{p}{p-q}(M-2)} \right].$$ 
Furthermore, considering $I_{M-k}$ with $k=1,2,...$, the interval of $\mu$ where $G^{[k]}$ admits a fixed point is such that $I_M\cap I_{M-k}=\emptyset.$ Moreover, their width are in
geometric progression in the sense that $|C_M| = \nu^{{\frac{pk}{p-q}}}|C_{M-1}|$, with $k=1,2,...$.
	
\begin{proof}
We know that the fixed point of $G$ exists for $\lambda=(\lambda_{PD},1].$ In terms of $\mu$ it is equivalent to $\mu\in(\mu_{PD},\mu_1]$ where $\mu_{PD}$ is given by Lemma \ref{mu_PD} and $\mu_1$ is given by equation \ref{lambda} with $\lambda$ substituted by $1.$

Now, from Lemma \eqref{mu_PD} we have $$\mu_{PD}=e^{\frac{p}{p-q}}(p+q)\left(\dfrac{q^q}{p^p}\right)^{\frac{1}{p-q}}\nu^{\frac{p }{p-q}(M-1)}.$$ On the other hand, it follows from \eqref{lambda} and \eqref{F} that
$$1=\lambda=\left(\dfrac{\nu}{\nu^{M-1}\mu_1^{\frac{q-p}{p}}e\sqrt[p]{(1-\nu)^q }}\right)^{p/q},$$ that implies that 
\begin{equation}\label{Iright}
\mu_1=e^{\frac{q}{p-q}}(1-\nu)^{\frac{q}{p-q}}\nu^{\frac{p}{p-q}(M-2)}.
\end{equation}

Therefore, $G$ admits a fixed point for
		$$\mu\in I_M=\left(\mu_{PD},\mu_1\right]=\left(e^{\frac{p}{p-q}}(p+q)\left(\dfrac{q^q} {p^p}\right)^{\frac{1}{p-q}}\nu^{\frac{p}{p-q}(M-1)},e^{\frac{q}{p-q}}(1-\nu)^{\frac{q}{p-q}}\nu^{\frac{p}{p-q}(M-2)}\right].$$
		
Similarly, by studying map $G^{[k]}$ we have that the interval associated with its fixed point is
$$\mu\in I_{M-k}=\left(e^{\frac{p}{p-q}}(p+q)\left(\dfrac{q^q} {p^p}\right)^{\frac{1}{p-q}}\nu^{\frac{p}{p-q}(M-1-k)},e^{\frac{q}{p-q}}(1-\nu)^{\frac{q}{p-q}}\nu^{\frac{p}{p-q}(M-2-k)}\right],$$
with $k=1,2,...$.

Since $1-\dfrac{q(p+q)^{p-q/q}}{p^{p+q}} < \nu < \dfrac{p}{p+q}$, these intervals do not overlap, that is, $I_M\cap I_{M-1} =\emptyset.$ This occurs because

$$e^{\frac{q}{p-q}}(1-\nu)^{\frac{q}{p-q}}\nu^{\frac{p}{p-q}(M-2)}<e^{\frac{p}{p-q}}(p+q)\left(\dfrac{q^q} {p^p}\right)^{\frac{1}{p-q}}\nu^{\frac{p}{p-q}(M-1-k)}\iff 1-\dfrac{q(p+q)^{p-q/q}}{p^{\frac{p}{q}}}<\nu<\nu^{-k},$$
for all $0<\nu<1$ and  $k=1,2,...$.

Furthermore, if $C_M$ and $C_{M-k}$ denotes respectively the length of the interval $I_M$ and $I_{M-k}$ with $k=1,2,...$, then 
$$
\begin{array}{rl}
\left|C_{M-k}\right|=&\nu^{\frac{-pk}{p-q}}\left(e^{\frac{q}{p-q}}(1-\nu)^{\frac{q}{p-q}}\nu^{\frac{p}{p-q}(M-2)}-e^{\frac{p}{p-q}}(p+q)\left(\dfrac{q^q} {p^p}\right)^{\frac{1}{p-q}}\nu^{\frac{p}{p-q}(M-1)}\right)\\
&\\
=&\nu^{\frac{-pk}{p-q}}\left|C_{M}\right|,	
\end{array}
$$
that means that its width are in geometric progression respecting
$$\left|C_{M}\right|=\nu^{\frac{pk}{p-q}}\left|C_{M-k}\right|,$$
with $k=1,2,...$
\end{proof}
	
\end{prop}	

Next lemma characterizes the dynamics of map $G$ given in \eqref{auxG} for $0<\nu<1-\dfrac{q(p+q)^{p-q/q}}{p^{\frac{p}{q}}}.$

\begin{lem} Let $G$ be defined in  \eqref{auxG}. If $0<\nu<1-\dfrac{q(p+q)^{p-q/q}}{p^{\frac{p}{q}}}$, then all the fixed points of $G$ are stable.\label{lemastable}
\end{lem}
\begin{proof}
	Computing the derivative of \eqref{derG} with respect to $z$ we obtain
	$$
	G_{zz}(\overline{z},\lambda)=\dfrac{q(q-p)\nu}{p^2\sqrt[p]{(1-\nu)^q}\sqrt[p]{(1-\overline{z})^{2p-q}}\lambda^{q/p}}<0, \qquad\forall\;z\in[\nu,1].
	$$
This means that $G_z(.,\lambda)$ is a decreasing function of $z.$ Now from \eqref{Gx} we know that $G_z\left(\dfrac{p}{p+q},\lambda\right)=-1.$ Therefore
$$G_z(\overline{z},\lambda)>-1, \forall\; \overline{z}\in \left[\nu,\dfrac{p}{p+q}\right)\; \mbox{and}\; \forall\; 0<\nu<1-\dfrac{q(p+q)^{(p-q)/q}}{p^{\frac{p}{q}}}.$$
Moreover, it follows from \eqref{derG}, that
$$G_z(z,\lambda)<0,\,\,\,\forall z\in [\nu,1].$$
Then, from the last two previous inequalities we get
$$
\left|G_z(\overline{ z},\lambda)\right|<1,\,\,\,\forall \;\overline{z} \in \left[\nu,\dfrac{p}{p+q}\right)\;\mbox{and}\;\forall\; 0<\nu<1-\dfrac{q(p+q)^{(p-q)/q}}{p^{\frac{p}{q}}}.
$$

This implies that $\overline{z}$ is a stable fixed point of $G$ and we get the result.
\end{proof}

Now we are in condition to prove the main result of this paper. It will be done in the next section.

\section{Proof of Theorem \ref{MainResult}}

\begin{proof} We consider the one-dimensional piecewise smooth rational degree map, $f:D\times\mathbb{R}\longrightarrow\mathbb{R}$, with $D\in\mathbb{R}$ given by
	
	\begin{equation}
	f(x,\mu)=\left\{\begin{array}{lr}
		f_1(x,\mu)=\nu x,& H(x,\mu)=x-\mu\geq0\\
		&\\
		f_2(x,\mu)=\nu x+e\sqrt[p]{(\mu-x)^q},& H(x,\mu)=x-\mu\leq0,
	\end{array}\right.	\label{funcf}
\end{equation}
where, $e>0$, $0<\nu<1$, $\mu\sim0^+,$ $p,q \in \mathbb{N}$ with $p>q$ and $(p,q)=1.$ Observe that, for $\mu^*\le0$ \eqref{funcf} has a stable fixed point at $x^*=0.$ 

We defined the trapping region
	$$W=\{x:\nu\mu \leq x\leq \mu\},$$
with $\mu\sim 0^+$, and the regions $R_I=\{x;\,\,x\leq\mu\},$ and $R_{II}=\{x;\,\,x\geq\mu\}.$

From Proposition \ref{proposicaom} it follows that if $x_0\in W$ is such that $x_0\approx\nu\mu$, then $f(x_0,\mu) \in R_{II}$.

Also from Proposition \ref{proposicaom} we define $m(x_0,\mu)$ as the first positive iteration of $x_0 \in W$ such that $f^m(x_0,\mu) \in W$. Taking  $z=\dfrac{x_0}{\mu}\in[\nu,1]$ we induce the map \eqref{F}
\begin{equation}
	\begin{array}{rll}
		F: [\nu,1] &\longrightarrow& W\\
		z & \mapsto &x_m=F(z)=F\left(\dfrac{x_0}{\mu}\right)=\nu^{m}\mu z+\nu^{m-1} e\sqrt[p]{\mu^q}\sqrt[p]{(1-z)^q},
	\end{array}
\end{equation}
with $z=\dfrac{x_0}{\mu}$, and $x_0\in W.$		
			
Now, as in Lemma \ref{lemaG} we take $k=M-m$, where $M=m(\nu\mu,\mu)$ is the maximum number of iterations given by Lemma \ref{lemaM}, and introduce the new parameter
	\begin{equation}
		\lambda=\left(\dfrac{\nu}{\dfrac{F(\nu)}{\mu}}\right)^{p/q}\in \left[\nu^{p/q},1\right],	
	\end{equation}
given in Lemma \ref{lemal}. By Lemma \ref{lemaG}, it follows that

\begin{equation}
		\begin{array}{ll}
			\dfrac{F(z)}{\mu}=&\dfrac{\nu^{1-k}\sqrt[p]{(1-z)^q}}{\sqrt[p]{(1-\nu)^q}\sqrt[p]{\lambda^q}}+\nu^{M-k}\bigg(z-\nu\dfrac{\sqrt[p]{(1-z)^q}}{\sqrt[p]{(1-\nu)^q}}\bigg)=G^{[k]}(z,\lambda)+\nu^{M-k}\bigg(z-\nu\dfrac{\sqrt[p]{(1-z)^q}}{\sqrt[p]{(1-\nu)^q}}\bigg),\\
		\end{array}\label{funff}
\end{equation}
with $k=0,1,2...$ and $z=\dfrac{x_0}{\mu}\in [\nu,1],$ where $G^{[0]}=G$ and $G$ is given in \eqref{auxG}.
	
	For $k=0,$ as $M$ can be made sufficiently large by taking $\mu\simeq0^+$ sufficiently small (Lemma \eqref{lemaM}), we have
	
	\begin{equation}
		\dfrac{F(z)}{\mu}\simeq G(z,\lambda),
	\end{equation}	
	where
	$$G:[\nu,1]\times\left[\nu^{p/q},1\right]\longrightarrow [\nu,1],$$
has the form
	\begin{equation}
		G(z,\lambda)=\dfrac{\nu\sqrt[p]{(1-z)^q}}{\sqrt[p]{(1-\nu)^q}\sqrt[p]{\lambda^q}}.
	\end{equation}	
	
	On the other hand, for $k\neq0$ we can write
	
	\begin{equation}
	\begin{array}{rl}
		\dfrac{F(z)}{\mu}=&\dfrac{\nu^{1-k}\sqrt[p]{(1-z)^q}}{\sqrt[p]{(1-\nu)^q}\sqrt[p]{\lambda^q}}+\nu^{M-k}\bigg(z-\nu\dfrac{\sqrt[p]{(1-z)^q}}{\sqrt[p]{(1-\nu)^q}}\bigg)\\
		&\\
		=&\dfrac{1}{\nu^k}\left(\dfrac{\nu\sqrt[p]{(1-z)^q}}{\sqrt[p]{(1-\nu)^q}\sqrt[p]{\lambda^q}}\right)+\nu^{M-k}\bigg(z-\nu\dfrac{\sqrt[p]{(1-z)^q}}{\sqrt[p]{(1-\nu)^q}}\bigg).\\
	\end{array}
\end{equation}
As $0<\nu<1$ for any $k\neq0$ and $\mu\simeq0^+$ sufficiently small it follows that
$$
\begin{array}{rl}
	\dfrac{F(z)}{\mu}\simeq G^{[k]}(z,\mu)=&\dfrac{\nu^{1-k}\sqrt[p]{(1-z)^q}}{\sqrt[p]{(1-\nu)^q}\sqrt[p]{\lambda^q}}.\\\
\end{array}
$$

Now, from Remark \eqref{obsG} we know that
	\begin{equation}
		G_z(z,\lambda)=-\dfrac{q}{p}\dfrac{\nu}{\sqrt[p]{(1-\nu)^q}\sqrt[p]{(1-z)^{p-q}}\lambda^{q/p}}<0,
	\end{equation}
and from \ref{funff} 
\begin{equation}\label{DGk}
	G_z^{[k]}(z,\lambda)=-\dfrac{q}{p}\dfrac{\nu^{1-k}}{\sqrt[p]{(1-\nu)^q}\sqrt[p]{(1-z)^{p-q}}\lambda^{q/p}}<0,
\end{equation}
Also from \eqref{funff} with $k=1,2,\ldots$ it follows that $G^{[k]}(z,\lambda)$ satisfy
	\begin{equation}
		G^{[k]}_z(z,\lambda)\leq G_z(z,\lambda)=-\dfrac{q\nu}{p\sqrt[p]{(1-\nu)^q}\sqrt[p]{(1-z)^{p-q}}\lambda^{q/p}}<0.
	\end{equation}

Moreover, $\overline{z}$ is a fixed point of $G^{[k]}$ if and only if 

\begin{equation}
	\dfrac{\overline{z}}{\sqrt[p]{(1-\overline{z})^q}}=\dfrac{\nu^{1-k}}{\sqrt[p]{(1-\nu)^q}\sqrt[p]{\lambda^q}}.\label{fixGk}
\end{equation}

From \eqref{DGk} and \eqref{fixGk} it follows that $G_{z}^{[k]}(\overline{z},\lambda)=-1$ if and only if
$$\overline{z}=\dfrac{p}{p+q}.$$	
	
Therefore, it follows that the dynamics of $G^{[k]}$ for $k=1,2,\ldots$ is qualitative analogous to the dynamics of $G.$ So, we can focus our analysis on the case where the number of iterations is maximum, that means that $m=M=m(\nu\mu,\mu)$ and $k=M-m=0$.

Therefore, it remains to study the possible bifurcation scenarios that may occur for $G(z,\lambda)$, with $\lambda \in \left[\nu^{p/q},1\right]$, in terms of $0<\nu<1.$

Related to function $G$ we have:
\begin{enumerate}
	\item[(i)] From Lemma \ref{lemachaos} if $\dfrac{p}{p+q}<\nu<1$ it admits a robust chaotic attractor close to the origin for all small values of $\mu\sim0^+;$
	\item[(ii)] From Lemma \ref{lemaperiod} if $1-\dfrac{q(p+q)^{p-q/q}}{p^{\frac{p}{q}}}<\nu<\dfrac{p}{p+q},$ then for all small values of $\mu\sim0^+,$ it admits  a period doubling bifurcation at $$\left(\overline{z},\lambda_{PD}\right)=\left (\dfrac{p}{p+q},\dfrac{q(p+q)^{\frac{p-q}{q}}}{p^{\frac{p}{q}}}\dfrac{ \nu^{\frac{p}{q}}}{1-\nu}\right).$$
It also follows from Lemma \eqref{mu_PD} that $\mu=\mu_{PD}$ and $\lambda_{PD}$ are related by	$\mu=\mu_{PD}=\left(\dfrac{\nu}{\nu^{M-1}e\sqrt[p]{(1-\nu)^q}\sqrt[p]{\lambda_{PD}^q}}\right)^{\frac{p}{q-p}}$, that implies that
$\mu_{PD}=e^{\frac{p}{p-q}}(p+q)\left(\dfrac{q^q}{p^p}\right)^{\frac{1}{p-q}}\nu^{\frac{p }{p-q}(M-1)}.$

Moreover, follows by the Proposition \ref{proposicaom} on which the $M$--periodic orbit exist in terms of $\mu$, that is, in terms of the parameter $\mu$, the stable fixed point of $G^{[0]}=G$ exists $\mu\in I_M$ where
$$\mu\in I_M=\left(\mu_{PD},\mu_1\right]=\left(e^{\frac{p}{p-q}}(p+q)\left(\dfrac{q^q}{p^p}\right)^{\frac{1}{p-q}}\nu^{\frac{p }{p-q}(M-1)},e^{\frac{q}{p-q}}(1-\nu)^{\frac{q}{p-q}}\nu^{\frac{p}{p-q}(M-2)} \right];$$

\item[(iii)]  From Lemma \ref{lemastable} if $0<\nu<1-\dfrac{q(p+q)^{p-q/q}}{p^{\frac{p}{q}}}$ then all the fixed points of $G$ are stable.
\end{enumerate}	
This concludes the proof of Theorem \ref{MainResult}.
\end{proof}

\section*{Appendix}\label{s2}

In this section we briefly describe the basic results from bifurcation theory  that we shall need for proving the main result of this paper.

Next result guarantees sufficient conditions for a family of one-parameter maps perform a \textit{period doubling bifurcation}. This result is useful for studying  family $G$ in \eqref{auxG} for $\nu \in (0,1).$ For a complete proof of this Theorem see \cite{reference12}.

\begin{teo}\label{FlipBifurcations}
Let $\overline{G}:\mathbb{R}\times\mathbb{R} \longrightarrow\mathbb{R}$ be a one-parameter family of maps such that, for $\lambda=\lambda_{PD},$ $\overline{G}$ has a fixed point $\overline{z}$ satisfying $\overline{G}_z(\overline{z}, \lambda_{PD})=-1$ with $K_1=\overline{G}_{\lambda}(\overline{z}, \lambda_{PD})\overline{G}_{zz}(\overline{z}, \lambda_{PD})+2\overline{G}_{\lambda z}(\overline{z}, \lambda_{PD})\neq0\,\,,
$ and $
K_2=\dfrac{1}{2}\overline{G}_{zz}(\overline{z}, \lambda_{PD})^2+\dfrac{1}{3}\overline{G}_{zzz}(\overline{z}, \lambda_{PD})^3\neq0,
$ Then:

\begin{itemize}
\item[(a)] There is a smooth curve of fixed points of $\overline{G}$ passing through $\left(\overline{z},\lambda_{PD}\right)$, where the stability changes in $\left(\overline {z},\lambda_{PD}\right)$.

\item[(b)] There is a smooth curve $\gamma$, passing through $\left(\overline{z},\lambda_{PD}\right)$, such that $\gamma-\{\left(\overline{z}, \lambda_{PD}\right)\}$ is the union of hyperbolic 2-periodic orbits.

\item[(c)] If $K_2>0$ (resp. $K_2<0$) then the periodic orbits are unstable (resp. stable).
\end{itemize}
\end{teo}

{\textbf{Acknowledgements.}} 
The first author is partially supported by Fapesp grant number
2019/10269-3. The second author were partially supported by a CEFET-MG: Centro Federal Tecnológico de Minas Gerais and Capes.

\end{section}
\bibliographystyle{plain}

\end{document}